\documentclass{amsart}

\usepackage{graphicx,color}
\usepackage{amssymb}
\usepackage{amscd}
\usepackage{longtable}
\usepackage{pb-diagram}
\usepackage[all]{xy}

\newtheorem{theorem}{Theorem}[section]
\newtheorem{proposition}[theorem]{Proposition}
\newtheorem{lemma}[theorem]{Lemma}

\newtheorem{problem}[theorem]{Problem}

\newtheorem{corollary}[theorem]{Corollary}
\newtheorem*{claim*}{Claim}

\theoremstyle{definition}
\newtheorem{definition}[theorem]{Definition}
\newtheorem{example}[theorem]{Example}

\newcommand{\N}{\Bbb N}
\newcommand{\F}{\Bbb F}

\newcommand{\C}{\Bbb C}
\newcommand{\Z}{\Bbb Z}

\newcommand{\D}{\Delta}

\newcommand{\la}{\langle}
\newcommand{\ra}{\rangle}

\newcommand{\ord}{\mathcal{O}}

\newcommand{\pr}{\mathrm{pr}}

\newcommand{\MV}{\mathcal{V}}
\newcommand{\K}{\mathcal{K}}
\newcommand{\ov}{\overline}

\theoremstyle{remark}
\newtheorem{remark}[theorem]{Remark}

\numberwithin{equation}{section}

\allowdisplaybreaks[4]

\begin{document}
\title{Twisted Alexander polynomials of a knot for group extensions} 

\author{Katsumi Ishikawa, Takayuki Morifuji, and Masaaki Suzuki}

\thanks{2020 {\it Mathematics Subject Classification}. 
Primary 57K14, Secondary 57K10, 20F38.}

\thanks{{\it Key words and phrases.\/}
Twisted Alexander polynomial, group extension, TAV group.}

\begin{abstract}
In this paper, we discuss twisted Alexander polynomials of a knot for group extensions of a finite group in two directions. Firstly, we provide a mod $p$ formula for the twisted Alexander polynomial of a knot in the $3$-sphere associated with the regular representation of a finite group. Secondly, we consider twisted Alexander polynomials of a knot for a series of central extensions of a finite group. Moreover, we apply these formulas for twisted Alexander polynomials to the study of twisted Alexander vanishing groups and orders for non-fibered knots. 
\end{abstract}

\address{Research Institute for Mathematical Sciences, Kyoto University, Kyoto 606-8502, Japan}
\email{katsumi@kurims.kyoto-u.ac.jp}

\address{Department of Mathematics, Hiyoshi Campus, Keio University, Yokohama 223-8521, Japan}
\email{morifuji@keio.jp}

\address{Department of Frontier Media Science, 
Meiji University, 4-21-1 Nakano, Nakano-ku, Tokyo 
164-8525, Japan}
\email{mackysuzuki@meiji.ac.jp}

\maketitle

\section{Introduction}

The twisted Alexander polynomial was originally introduced by Lin \cite{Lin01-1} for knots in the $3$-sphere $S^3$, and by Wada \cite{Wada94-1} for finitely presentable groups. It is defined for the pair of a group (e.g., a knot group) and its linear representation, and 
gives a natural generalization of the classical Alexander polynomial. The theory of twisted Alexander polynomials has developed over the last thirty years. In fact, we can exract more various information of knots and  $3$-manifolds from twisted Alexander polynomial than Alexander polynomial. For example, it is known that fibered $3$-manifolds are detected by the set of twisted Alexander polynomials associated with the regular representations of finite groups (see \cite{FV11-1}). 
However, explicit formulas for twisted Alexander polynomials are not well known except for special classes of knots and their representations. 

In this paper, we derive formulas for  twisted Alexander polynomials of a knot for two kinds of group extensions of a finite group. Firstly, we  provide a mod $p$ formula for the twisted Alexander polynomial $\D_K^{\rho\circ f}(t)$ of a knot $K$ in $S^3$ associated with the composition of an epimorphism $f\colon G(K)\to G$ and the regular representation $\rho\colon G\to \mathrm{GL}(|G|,\Z)$ of a finite group $G$, where $G(K)$ is the fundamental group of the exterior of $K$ in $S^3$. By `mod $p$' we mean the reduction modulo $p$ of each coefficient of $\D_K^{\rho\circ f}(t)$. For a normal subgroup $H$ of $G$, we can show the following. 

\begin{theorem}\label{thm:main-1}
Let $H \subset G$ be a normal subgroup of order $p^n$ for a prime number $p$. Then for an epimorphism $f\colon G(K)\to G$, we have
$$\Delta_K^{\rho\circ f}(t) \equiv \left( \Delta_K^{\tilde{\rho}\circ f}(t) \right)^{p^n} ~(\mathrm{mod} ~p),$$
where 
$\tilde{\rho}=\rho_{H\backslash G}\circ\pi$ is the composition of 
the quotient map $\pi \colon G \to H\backslash G$ and the regular representation $\rho_{H\backslash G} \colon H\backslash G \to {\rm GL}(|H \backslash G|, \mathbb{Z})$ of the quotient group $H\backslash G$. 
\end{theorem}

If we apply Theorem \ref{thm:main-1} to the dihedral group $D_{p^n}=C_{p^n}\rtimes C_2$ of order $2p^n$ for an odd prime number $p$, we obtain the following explicit formula for $\D_K^{\rho\circ f}(t)$. Here, $C_m$ is a cyclic group of order $m$, and let $\D_K(t)$ be the Alexander polynomial of $K$. 

\begin{corollary}\label{cor:dihedral}
For a given epimorphism $f\colon G(K) \to D_{p^n}$, we have
$$
\D_{K}^{\rho\circ f}(t)
\equiv
\left(
\frac{\D_K(t)}{t-1}\cdot
\frac{\D_K(-t)}{t+1}
\right)^{p^n}
~(\mathrm{mod}~ p).
$$
\end{corollary}

This formula is a generalization of a result due to Boden and Friedl \cite[Theorem 2(4)]{BF14-1} in the sense that when $n=1$ they just provided the above mod $p$ formula for the twisted Alexander polynomial  of a knot $K$ associated with the permutation representation of the symmetric group $S_p$. We also derive similar formulas for several interesting finite groups in Section \ref{sec:modp}. 

Secondly, let us consider a central extension 
$1\rightarrow C_n \rightarrow G_{k,n} \overset{\pr}{\longrightarrow} G_{k,1} \rightarrow 1$ by a finite cyclic group $C_n$, 
where $G_{k,n}$ is defined as the pull-back of two epimorphisms $\pi\colon G_{k,1}\to C_k$; the abelianization homomorphism, and $\pi_k\colon C_{kn}\to C_k$; a natural projection, where $k$ and $n$ are positive integers (see Section \ref{sec:extension} for details). 
We denote the regular representation of the group $G_{k,n}$ by $\rho_n\colon G_{k,n} \to\mathrm{GL}(kn|H|,\Z)$, where $H$ is the kernel of $\pi$, and the composition $\rho_1\circ \mathrm{pr}$ by $\tilde{\rho}_1$. When there exists an epimorphism $f_n\colon G(K) \to G_{k,n}$, we have the following formula. 

\begin{theorem}\label{thm:main-2}
$\displaystyle{
\D_K^{\rho_n\circ f_n}(t)
=
\prod_{j=0}^{n-1}
\D_K^{\tilde{\rho}_1 \circ f_n}\big(e^{\frac{2\pi i}{kn}j}t\big).}$
\end{theorem}

Namely, the twisted Alexander polynomial of the $n$-th central extension $G_{k,n}$ is determined by that of $G_{k,1}$. As we will explain below, the formula leads an important consequence with respect to TAV groups and TAV orders. 

Another motivation of this paper comes from applications of these formulas to the study of twisted Alexander vanishing groups and orders of knots introduced in \cite{IMS23-1}. We call a finite group $G$ a \textit{twisted Alexander vanishing (TAV) group of a knot $K$} if there exists an epimorphism $f\colon G(K)\to G$ such that $\D_K^{\rho\circ f}(t)$ is zero. By abuse of terminology, we also call a finite group $G$ a \textit{TAV group} if $G$ is a TAV group of some knot $K$. 

A knot $K$ is called  \textit{fibered} if the complement of $K$ in $S^3$ admits a structure of a surface bundle over the circle such that the closures of the fiberes are Seifert surfaces. Then a vanishing theorem for non-fibered knots due to Friedl and Vidussi \cite[Theorem 1.2]{FV13-1} implies that every non-fibered knot admits a TAV group. 

Next, we define the \textit{twisted Alexander vanishing order} $\ord(K)$ of a non-fibered knot $K$ to be the order of the smallest TAV group of $K$, 
and call it the \textit{TAV order} of $K$ in short (we call it the \textit{minimal order} of $K$ in \cite{MS22-1}). For a fibered knot $K$, we set $\ord(K)=+\infty$, because its twisted Alexander polynomial is monic 
(see \cite{Cha03-1}, \cite{GKM05-1}), 
and hence never vanishes. 
Thus, for the set $\K$ consisting of oriented knots in $S^3$, we can obtain the \textit{TAV order function} $\ord\colon\K\to\N\cup\{+\infty\}$. 

In \cite{IMS23-1}, we provided a characterization of a TAV group. In order to state the result precisely, 
we recall some terminologies in group theory. 
The \textit{weight} of a group $G$, denoted by $w (G)$, is the smallest integer $n$ such that $G$ is the normal closure of $n$ elements. We also set $w(\{e\})=0$. For a group $G$, there are a knot $K$ and an epimorphism $f\colon G(K)\to G$ if and only if $G$ is finitely generated and $w(G)\leq1$ (see \cite{GA75-1}, \cite{Johnson80-1}). A finite group $G$ is a \textit{$p$-group} if and only if the order $|G|$ is a power of a prime number. We denote the commutator subgroup of $G$ by $G'$. Then we have:

\begin{theorem}[{\cite[Theorem~1.5]{IMS23-1}}]\label{thm:main-4}
A finite group $G$ is a TAV group if and only if 
$w(G)=1$ and $G'$ is not a $p$-group. 
\end{theorem}

As a corollary of Theorem \ref{thm:main-1}, we can provide a short proof of `only if' part of Theorem \ref{thm:main-4}. Note that if $w(G)\not=1$, then there is no epimorphism $f\colon G(K)\to G$ for any knot $K$, or $G$ is the trivial group $\{e\}$. 

\begin{corollary}\label{cor:onlyif}
Let $G$ be a finite group of weight one such that $G'$ is a $p$-group. Then for an epimorphism $f\colon G(K)\to G$, the reduction modulo $p$ of $\D_K^{\rho\circ f}(t)$ is described using the Alexander polynomial of $K$. In particular, $\D_K^{\rho\circ f}(t)$ is nonzero. 
\end{corollary}

On the other hand, using Theorem \ref{thm:main-2}, we can derive  some information on the image of the TAV order function $\ord|_{\mathcal{N}}\colon \mathcal{N}\to \N$, where $\mathcal{N}\subset \K$ is the set of non-fibered knots. More precisely, there exists a positive integer which is the order of a TAV group, but never realize the TAV order of any non-fibered knot. As a partial answer to a problem proposed in \cite{IMS23-1}, we have: 

\begin{corollary}\label{cor:image}
The set of orders of TAV groups contains $\mathrm{Im}\,\ord|_{\mathcal{N}}$ as a proper subset.
\end{corollary}

This paper is organized as follows. 
In Section \ref{sec:2}, we recall the definition of twisted Alexander polynomial, and provide several examples of both TAV groups and non-TAV groups. 
In Section \ref{sec:modp}, we show Theorem \ref{thm:main-1}, Corollaries \ref{cor:dihedral} and \ref{cor:onlyif}. Moreover, we give alternative proof of \cite[Theorem 1.2(i)]{IMS23-1} concerning the minimum value of the TAV order function $\ord\colon\K\to\N\cup\{+\infty\}$. 
In Section \ref{sec:extension}, we prove Theorem \ref{thm:main-2}, and introduce the notion of seed for finite groups of weight one. 
In Section \ref{sec:remark}, we show Corollary \ref{cor:image}, namely, the existence of an order of TAV groups which cannot realize the TAV order of any non-fibered knot. 

A part of this paper is contained in our unpublished manuscripts \cite{IMS-II}, \cite{MS23-1}. In this paper, we develop it from the viewpoint of the formulation of twisted Alexander polynomials of a knot for group extensions. 
 
\section{Preliminaries}\label{sec:2}

\subsection{Twisted Alexander polynomials}\label{subsec:2.1}

In this paper, we consider an oriented knot $K$ in the $3$-sphere $S^3$ and let $E_K=S^3\setminus \nu(K)$, where $\nu(K)$ denotes an open tubular neighborhood of $K$. 
We denote $\pi_1(E_K)$ by $G(K)$, 
and call it the \textit{knot group} of $K$. 
We fix a Wirtinger presentation $G(K)=\la s_1,\ldots,s_n\,|\,r_1,\ldots,r_{n-1}\ra$ of the knot group, and let $\phi\colon G(K)\to H_1(E_K;\Z)=\Z=\la t\ra$ be the abelianization, that is, $\phi$ maps an oriented meridian $\mu_K$ of $K$ to the generator $t$. 

Let $R$ be a Noetherian unique factorization domain with quotient field $Q$. 
For a given representation $\varrho\colon G(K)\to \mathrm{GL}(k,R)$, we extend the group homomorphism $\phi\otimes\varrho\colon G(K)\to \mathrm{GL}(k,R[t^{\pm1}])$, 
which is defined by $(\phi\otimes\varrho)(s)=\phi(s)\varrho(s)$ for $s\in G(K)$, 
to a ring homomorphism $\Z[G(K)]\to \mathrm{Mat}(k,R[t^{\pm1}])$, 
where the target is the matrix algebra over $R[t^{\pm1}]$. 
Let $F_n$ denote the free group $\la s_1,\ldots,s_n\ra$ 
and $\Phi\colon \Z[F_n]\to \mathrm{Mat}(k,R[t^{\pm1}])$ the composition of the surjection $\Z[F_n]\to\Z[G(K)]$ induced by the presentation of $G(K)$ and the ring homomorphism $\Z[G(K)]\to \mathrm{Mat}(k,R[t^{\pm1}])$. 

Let $M=(m_{ij})$ be the $k(n-1)\times kn$ matrix with 
$m_{ij}=\Phi(\frac{\partial r_i}{\partial s_j})$, 
where $\frac{\partial}{\partial s_j}$ is the free differential by $s_j$. 
For $1\leq j\leq n$ we denote by $M_j$ the $k(n-1)\times k(n-1)$ matrix obtained from $M$ by removing the $j$-th block column. 
Then Wada \cite{Wada94-1} defined the twisted Alexander polynomial $\D_K^\varrho(t)$ associated with $\varrho\colon G(K)\to\mathrm{GL}(k,R)$ to be the rational expression 
$$
\D_K^\varrho(t)
=
\frac{\det M_j}{\det\Phi(s_j-1)}\in Q(t),
$$
which is well defined up to multiplication by $\varepsilon t^{l}~(\varepsilon\in R^\times,\,l\in\Z)$. The following lemma is elementary (see \cite[Lemma 5]{BF14-1}). 

\begin{lemma}\label{lem:BF}
Let $K$ be an oriented knot with the Alexander polynomial $\D_K(t)$. 
\begin{itemize}
\item[(i)]
If $\tau\colon G(K)\to \mathrm{GL}(1,\C)$ is a representation which is given by sending $\mu_K$ to $\xi\in\C\setminus\{0\}$, then $\D_K^\tau(t)=\D_K(\xi t)/(\xi t-1)$ holds. 

\item[(ii)]
If $\varrho\colon G(K)\to \mathrm{GL}(k,R)$ is a representation of the form $\varrho(s)=
\begin{pmatrix}
\beta(s)&\gamma(s)\\
0&\delta(s)
\end{pmatrix}$, then $\D_K^\varrho(t)=\D_K^\beta(t)\cdot\D_K^\delta(t)$ holds. 
\end{itemize}
\end{lemma}

Let $\rho\colon G \to \mathrm{Aut}_\Z(\Z[G])$ be the regular representation of a finite group $G$, which is a permutation representation that corresponds to the action of $G$ on itself through right multiplication. We can also identify $\mathrm{Aut}_\Z(\Z[G])$ with $\mathrm{GL}(|G|,\Z)$. Thus, for an epimorphism $f \colon G(K) \to G$, we can consider the twisted Alexander polynomial $\Delta_K^{\rho \circ f}(t)$ associated with the representation $\rho\circ f\colon G(K)\to \mathrm{GL}(|G|,\Z)$. 

\subsection{Twisted Alexander vanishing groups}\label{subsec:2.2}

For reader's convenience, we exhibit here some examples of both TAV groups and non-TAV groups. 

First of all, a finite \textit{abelian group} $G$ of weight one is not a TAV group, because it
is a cyclic group $C_m$. Using Lemma \ref{lem:BF}(i), for an epimorphism $f\colon G(K)\to C_m$, we obtain 
\begin{equation}\label{eq:cyclic}
\D_K^{\rho\circ f}(t)=\prod_{j=0}^{m-1}\frac{\D_K(\alpha^jt)}{\alpha^jt-1},
\end{equation}
where $\alpha\in \C$ is a primitive $m$-th root of unity (see \cite[Example 2.6]{IMS23-1}). Because the image $\rho(\mu_K)$ of an oriented meridian $\mu_K$ is conjugate to the diagonal matrix $\mathrm{diag}(1,\alpha,\ldots,\alpha^{m-1})$. Hence, $\D_K^{\rho\circ f}(t)$ is nonzero for any knot $K$, because so is $\D_K(t)$. 

Next, every \textit{$p$-group} is not a TAV group. Since a finite $p$-group $G$ is nilpotent, $G$ is cyclic when $w(G)=1$. Thus, a $p$-group $G$ is not a TAV group. 

On the other hand, the following three kinds of groups are TAV groups: 

\begin{itemize}
\item
The \textit{alternating group} $A_n~(n\geq5)$. 

\item
The \textit{symmetric group} $S_n=A_n\rtimes C_2~(n\geq4)$.

\item
The \textit{dihedral group} $D_n=C_n\rtimes C_2$ of order $2n$, where $n$ is odd and not a power of an odd prime number. 
\end{itemize}

For $n\geq2$, the \textit{dicyclic group} $\mathrm{Dic}_n$ of order $4n$ is defined by the presentation $\mathrm{Dic}_n=
\la
a,b\,|\,a^{2n},b^2a^{-n},bab^{-1}a
\ra$, 
and is often called the \textit{binary dihedral group}, because $\mathrm{Dic}_n/\la y^2\ra$ is isomorphic to $D_n$. If $n$ is even, $\mathrm{Dic}_n$ is not of weight one, and hence not a TAV group. Then we have: 

\begin{itemize}
\item
The \textit{dicyclic group} $\mathrm{Dic}_n$ of order $4n$, where 
$n$ is odd and not a power of an odd prime number, is a TAV group. 
\end{itemize}

\section{A mod $p$ formula for group extensions}\label{sec:modp}

The following formula is our first main theorem of this paper, which includes Theorem \ref{thm:main-1} in the introduction. Let  $G$ be a finite group of weight one, and $\rho\colon G\to\mathrm{GL}(|G|,\Z)$ its regular representation.  

\begin{theorem}\label{modp-thm}
Let $H \subset G$ be a subgroup of order $p^n$ for a prime number $p$. Then for an epimorphism $f\colon G(K) \to G$, 
we have
$$\Delta_K^{\rho\circ f}(t) \equiv \left(\Delta_K^{\overline{\rho} \circ f}(t)\right)^{p^n} ~(\mathrm{mod}~ p),$$
where $\overline{\rho} \colon G \to {\rm GL}(|H \backslash G|, \mathbb{Z})$ is the 
permutation representation associated to the right action of $G$ on the cosets $Hg~(g\in G)$. In particular, if $H$ is a normal subgroup of $G$, we have
\begin{equation*}\label{eq:thm}
\Delta_K^{\rho\circ f}(t) \equiv \left( \Delta_K^{\tilde{\rho}\circ f}(t) \right)^{p^n} ~(\mathrm{mod}~ p),
\end{equation*}
where 
$\tilde{\rho}=\rho_{H\backslash G}\circ\pi$ is the composition of 
the quotient map $\pi \colon G \to H\backslash G$ and the regular representation $\rho_{H\backslash G} \colon H\backslash G \to {\rm GL}(|H \backslash G|, \mathbb{Z})$ of the quotient group $H\backslash G$. 
\end{theorem}

\begin{remark}
\begin{itemize}
\item[(i)]
Note that $\big(\Delta_K^{\overline{\rho} \circ f}(t)\big)^{p^n}\equiv\,\Delta_K^{\overline{\rho} \circ f}(t^{p^n})~(\mathrm{mod}~p)$ holds.
\item[(ii)]
The theorem also holds for a finite CW complex $X$ and $\phi\in H^1(X;\Z)=\mathrm{Hom}(\pi_1(X),\Z)$ an epimorphism, which is non-necessarily abelianization. But we omit the detail here. 
\end{itemize}
\end{remark}

Before proving Theorem \ref{modp-thm}, let us see an example. 

\begin{example}
Let $G_1=\mathrm{CSU}_2(\F_3)=Q_8.S_3=C_2._{\overset{\phantom{ }}{2}}S_4$ and $G_2=\mathrm{GL}_2(\F_3)=Q_8\rtimes S_3=C_2._{\overset{\phantom{ }}{2}}S_4$ be groups of weight one of order $48$ (see \cite{GN} for the precise definition of these groups). Since both the quaternion group $Q_8$ and the cyclic group $C_2$ are normal subgroups of $G_1$ and $G_2$, we can apply Theorem \ref{modp-thm} to these groups. Let $f_i\colon G(K) \to G_i~(i=1,2)$ be epimorphisms and $\rho_i\colon G_i\to \mathrm{GL}(48,\Z)$ the regular representations of $G_i$. Then we have
$$
\Delta_K^{\rho_i\circ f_i}(t) 
\equiv 
\left( \Delta_K^{\tilde{\rho}_{S_3}\circ f_i}(t) \right)^{2^3} 
\equiv
\left( \Delta_K^{\tilde{\rho}_{S_4}\circ f_i}(t) \right)^{2} ~(\mathrm{mod}~ 2).
$$
These two descriptions of $\Delta_K^{\rho_i\circ f_i}(t)$ come from the group extension $S_4=C_2^2\rtimes S_3$. In fact, applying Theorem \ref{modp-thm} to this extension, we have  
$$
\Delta_K^{\rho_{S_4}\circ f}(t)
\equiv
\left( \Delta_K^{\tilde{\rho}_{S_3}\circ f}(t) \right)^{2^2} ~(\mathrm{mod}~ 2)
$$
for an epimorphism $f\colon G(K)\to S_4$, the regular representation $\rho_{S_4}\colon S_4\to\mathrm{GL}(24,\Z)$, and $\tilde{\rho}_{S_3}=\rho_{S_3}\circ\pi\colon S_4\to \mathrm{GL}(6,\Z)$. 
\end{example}

We need the following two lemmas to show Theorem \ref{modp-thm}. These results may be well known to experts, but in order to make this paper as self-contained as possible, we will provide their proofs here (see \cite[Sections I, III]{Alperin86-1}). 

Let $\F_p$ denote the prime field of characteristic $p>0$. 

\begin{lemma}\label{modp-lem}
Let $V \cong \mathbb{F}_p[C_p^n]$ be the regular representation of $C_p^n$ with coefficient $\mathbb{F}_p$. Then there exists a sequence
$$V = V_0 \supset V_1 \supset \cdots \supset V_{n(p-1) + 1} = 0$$
of subrepresentations such that $V_j / V_{j+1}$ is a trivial representation for each $j$.
\end{lemma}
\begin{proof}
Taking generators $x_1, \dots, x_n \in C_p^n$, we identify $\mathbb{F}_p[C_p^n]$ with 
$$
\mathbb{F}_p[x_1, \dots, x_n]/(x_1^p-1, \dots, x_n^p - 1) = \mathbb{F}_p[x_1, \dots, x_n]/((x_1-1)^p, \dots, (x_n-1)^p). 
$$ 
Replacing $x_i -1$ with $y_i$, we find $V$ is isomorphic to $\mathbb{F}_p[y_1,\dots,y_n]/(y_1^p, \dots, y_n^p)$. Define $V_j \subset \mathbb{F}_p[y_1,\dots,y_n]/(y_1^p, \dots, y_n^p)$ to be the subspace generated by monomials with total degree greater than or equal to $j$:
$$V_j = {\rm span}\{ y_1^{a_1} \cdots y_n^{a_n} \mid 0 \leq a_i < p,\; a_1 + \cdots + a_n \geq j\}.$$
Because the action of $x_i$ on $\mathbb{F}_p[y_1,\dots,y_n]/(y_1^p, \dots, y_n^p)$ is the multiplication by $1 + y_i$, $V_j$ is a subrepresentation and $(x_i - 1) V_j \subset V_{j+1}$, i.e., the induced action on $V_j / V_{j+1}$ is trivial.
\end{proof}

\begin{lemma}\label{pgp-lem}
Let $H$ be a finite $p$-group and $V \cong \mathbb{F}_p[H]$ the regular representation of $H$ with coefficient $\mathbb{F}_p$. Then there exists a sequence
$$V = V_0 \supset V_1 \supset \cdots \supset V_N = 0$$
of subrepresentations such that $V_j / V_{j+1}$ is a trivial representation for each $j$.
\end{lemma}
\begin{proof}
We show the lemma by induction on the order of $H$; it is trivial if $|H| = 1$.
\par Let $|H|$ be a finite $p$-group and assume that the lemma holds for any $p$-group with order less than $|H|$. Let $p^a$ be the maximum of the orders of the elements of the abelianization $H/H'$. Since $H$ is nilpotent, $H/H'$ is nontrivial and hence $a > 0$. We define a subgroup $H_0$ of $H$ by
$$H_0 = \{h \in H \mid h^{p^{a-1}} \in H'\}.$$
We find that $H_0$ is normal and $H/H_0$ is isomorphic to $C_p^n$ for some $n > 0$. Since $|H_0| < |H|$, we can apply the assumption to $H_0$ and there exists a sequence
$$\mathbb{F}_p[H_0] = W_0 \supset W_1 \supset \cdots \supset W_{N_0} = 0$$
of subrepresentations of $H_0$ such that $W_j/W_{j+1}$ is trivial for any $j$. Recall that the regular representation of $H$ is induced from that of $H_0$, i.e., $\mathbb{F}_p[H] \cong  \mathbb{F}_p[H_0] \otimes_{\mathbb{F}_p[H_0]} \mathbb{F}_p[H]$. Thus, we obtain a sequence
\begin{equation}\label{eq1}
\mathbb{F}_p[H] \cong W_0 \otimes_{\mathbb{F}_p[H_0]} \mathbb{F}_p[H] \supset \cdots \supset W_{N_0} \otimes_{\mathbb{F}_p[H_0]} \mathbb{F}_p[H] = 0
\end{equation}
of subrepresentations of $H$. Since $H_0$ is normal, the action of $H_0$ on $(W_j \otimes_{\mathbb{F}_p[H_0]} \mathbb{F}_p[H])/(W_{j+1} \otimes_{\mathbb{F}_p[H_0]} \mathbb{F}_p[H]) \cong (W_j/W_{j+1}) \otimes_{\mathbb{F}_p[H_0]} \mathbb{F}_p[H]$ is trivial and we can regard $(W_j \otimes_{\mathbb{F}_p[H_0]} \mathbb{F}_p[H])/(W_{j+1} \otimes_{\mathbb{F}_p[H_0]} \mathbb{F}_p[H])$ as a representation of $H/H_0$; we can apply Lemma \ref{modp-lem} to this quotient representation. Therefore we can take a subdivision of the sequence (\ref{eq1}) to obtain a required one.
\end{proof}

\begin{proof}[Proof of Theorem \ref{modp-thm}]
We take a sequence of subrepresentations $\mathbb{F}_p[H] = V_0 \supset V_1 \supset \cdots \supset V_N = 0$ of $H$ as in Lemma \ref{pgp-lem} and let $W_j$ be the representation of $G$ induced from $V_j$, i.e., $W_j = V_j \otimes_{\mathbb{F}_p[H]} \mathbb{F}_p[G]$; since the regular representation of $G$ is induced from the regular representation of $H \subset G$, we obtain a sequence $\mathbb{F}_p[G] = W_0 \supset W_1 \supset \cdots \supset W_N = 0$ of representations of $G$. Furthermore, since $V_j / V_{j+1}$ is trivial, $W_j / W_{j+1} = (V_j / V_{j+1}) \otimes_{\mathbb{F}_p[H]} \mathbb{F}_p[G]$ is isomorphic to the direct sum of copies of the 
permutation representation of $G$ on $H \backslash G$. Thus, we can represent the $\mathbb{F}_p$-coefficient regular representation of $G$ in block triangular matrices so that each of the $p^n$ diagonal blocks represents the 
permutation representation on $H \backslash G$. By Lemma \ref{lem:BF}(ii), we obtain the required formula. 
\end{proof}

Here, let us prove Corollary \ref{cor:dihedral} in the introduction. 

\begin{corollary}[{Corollary \ref{cor:dihedral}}]\label{cor:dihedral2}
For a given epimorphism $f\colon G(K) \to D_{p^n}$, we have
$$
\D_{K}^{\rho\circ f}(t)
\equiv
\left(
\frac{\D_K(t)}{t-1}\cdot
\frac{\D_K(-t)}{t+1}
\right)^{p^n}
~(\mathrm{mod}~ p).
$$
\end{corollary}

\begin{proof}
Since the commutator subgroup of $D_{p^n}=C_{p^n}\rtimes C_2$ is $C_{p^n}$, we can apply Theorem \ref{modp-thm} to $H=C_{p^n}$. Using (\ref{eq:cyclic}) for the case $m=2$ (i.e. $\alpha=-1$), we obtain the desired formula.  
\end{proof}

Applying Theorem \ref{modp-thm} to the commutator subgroup $G'$ of a finite group $G$ of weight one, we obtain Corollary \ref{cor:onlyif}. Namely, the reduction modulo $p$ of the twisted Alexander polynomial is expressed in terms of the Alexander polynomial. 

In order to prove Corollary \ref{cor:onlyif}, let us first consider the cyclic group $C_m$ where $m=p^kl$, $p$ is a prime number, and $p \nmid l$. Let $\F_{p^d}$ be the finite field where $d$ is the smallest positive integer such that $p^d\equiv1~(\mathrm{mod}~l)$; i.e., $d=\mathrm{ord}_l(p)$. For a polynomial $P(t)\in\Z[t^{\pm1}]$, we denote by $\ov{P}(t)$ the polynomial obtained from $P(t)$ by reducing each coefficient modulo $p$. Note that we can regard $\ov{P}(t)$ as an element of $\F_{p^d}[t^{\pm1}]$ by identifying $\F_p$ with its image in $\F_{p^d}$. 

\begin{proposition}\label{pro:cyclic}
For an epimorphism $f\colon G(K)\to C_m$ and the regular representation $\rho\colon C_m\to\mathrm{GL}(m,\Z)$, we have
$$
\ov{\D_K^{\rho\circ f}}(t)
=
\left(\prod_{j=0}^{l-1}\frac{\ov{\D_K}(\zeta^jt)}{\zeta^jt-1}\right)^{p^{k}}
$$
in $\F_p(t)$, where $\zeta$ is a primitive $l$-th root of unity in $\F_{p^d}$. In particular, $\ov{\D_K^{\rho\circ f}}(t)$ is nonzero. 
\end{proposition}

\begin{proof}
As $p \nmid l$, it follows that $C_m\cong C_{p^k}\times C_l$, and thus we can apply Theorem \ref{modp-thm} with $H=C_{p^k}$ and $H\backslash G=C_l$. Moreover, Maschke's theorem implies that the regular representation $\rho_{C_l}\colon C_l\to\mathrm{GL}(l,\Z)$ is diagonalizable over $\F_{p^d}$, and hence the image $\tilde{\rho}(\mu_K)$ is conjugate to $\mathrm{diag}(1,\zeta,\ldots,\zeta^{l-1})$. Therefore we obtain 
\begin{equation}\label{eq:cyclic4}
\ov{\D_K^{\tilde{\rho}\circ f}}(t)=\prod_{j=0}^{l-1}\frac{\ov{\D_K}(\zeta^jt)}{\zeta^jt-1}
\end{equation}
in $\F_p(t)$. Note that each $\ov{\D_K}(\zeta^j t)/(\zeta^j t-1)$ is an element of $\F_{p^d}(t)$, but the product $\prod_{j=0}^{l-1}\ov{\D_K}(\zeta^jt)/(\zeta^jt-1)$ is included in $\F_p(t)$. Using (\ref{eq:cyclic4}) and Theorem \ref{modp-thm}, we derive the required formula. The latter assertion follows from $\ov{\D_K}(\zeta^it)\not=0$ in $\F_{p^d}[t^{\pm1}]$. 
\end{proof}

\begin{remark}
Proposition \ref{pro:cyclic} also follows without using Theorem \ref{modp-thm}, since $\F_{p^d}[C_m]\cong\F_{p^d}[C_{p^k}]\otimes_{\F_{p^d}}\F_{p^d}[C_l]$ for $m=p^kl$ with $p\nmid l$, where $d=\mathrm{ord}_l(p)$. 
\end{remark}

Now, we give the precise statement of Corollary \ref{cor:onlyif} in the introduction. 

\begin{corollary}\label{cor:modp-TAP}
Let $G$ be a finite group of weight one such that $|G'|=p^n$ and $G/G'= C_m$ where $p$ is a prime number, $m=p^kl$, and $p\nmid l$. Then for an epimorphism $f\colon G(K) \to G$ and the regular representation $\rho\colon G\to \mathrm{GL}(mp^n,\Z)$, we have 
$$
\ov{\D_K^{\rho\circ f}}(t)
=
\left(\prod_{j=0}^{l-1}\frac{\ov{\D_K}(\zeta^jt)}{\zeta^jt-1}\right)^{p^{k+n}}
$$
in $\F_p(t)$, where $\zeta$ is a primitive $l$-th root of unity in $\F_{p^d}$ and $d=\mathrm{ord}_l(p)$. In particular, $\D_K^{\rho\circ f}(t)$ is nonzero. 
\end{corollary}

\begin{proof}
This is an immediate consequence of Theorem \ref{modp-thm} and Proposition \ref{pro:cyclic}. 
\end{proof}

Let us consider some examples. In each example, $f\colon G(K)\to G$ is an epimorphism, and $\rho\colon G\to\mathrm{GL}(|G|,\Z)$ denotes the regular representation. 

\begin{example}\label{ex:A4}
Let $G=A_4=C_2^2\rtimes C_3$, where $G'=C_2^2$ and $G/G'=C_3$. Using Corollary \ref{cor:modp-TAP},  we have 
\begin{equation*}
\ov{\D_K^{\rho\circ f}}(t)=\left(\prod_{j=0}^{2}\frac{\ov{\D_K}(\zeta^jt)}{\zeta^jt-1}\right)^4
\end{equation*}
in $\F_2(t)$, where $\zeta$ is a primitive $3$-rd root of unity in $\F_4$. 
\end{example}

\begin{example}\label{ex:Dic}
Let $G=\mathrm{Dic}_q$ where $q=p^n$ and $p$ is an odd prime number. The dicyclic group $\mathrm{Dic}_q$ fits the following exact sequence: $1\to C_q\to \mathrm{Dic}_q\to C_4\to1$. Using Corollay \ref{cor:modp-TAP}, we have 
$$
\ov{\D_K^{\rho\circ f}}(t)=
\left(\prod_{j=0}^{3}\frac{\ov{\D_K}(\zeta^jt)}{\zeta^jt-1}\right)^q
$$
in $\F_p(t)$, where $\zeta$ is a primitive $4$-th root of unity in $\F_{p^d}$. 
\end{example}

\begin{example}\label{ex:Gmpk}
Let $G=G(m,p|\zeta)=\la a,b\,|\, a^p=b^m=1,bab^{-1}=a^\zeta\ra$ which is isomorphic to the semi-direct product $C_p\rtimes C_m$. Here, $m\in\N$, $p$ is an odd prime number such that $p\equiv1~(\mathrm{mod}~m)$, and $\zeta$ is a primitive $m$-th root of unity in $\F_p$. 
Using Corollary \ref{cor:modp-TAP}, we have 
\begin{equation*}
\ov{\D_K^{\rho\circ f}}(t)=\left(\prod_{j=0}^{m-1}\frac{\ov{\D_K}(\zeta^jt)}{\zeta^jt-1}\right)^p
\end{equation*}
in $\F_p(t)$.
\end{example}

\begin{example}
Let $G=C_3\times \mathrm{Dic}_3=C_3\rtimes C_{12}$, where $G'=C_3$ and $G/G'=C_{12}$. Using Corollary \ref{cor:modp-TAP}, we have 
$$
\ov{\D_K^{\rho\circ f}}(t)
=
\left(\prod_{j=0}^{3}\frac{\ov{\D_K}(\zeta^jt)}{\zeta^jt-1}\right)^{9}
$$
in $\F_3(t)$, where $\zeta$ is a primitive $4$-th root of unity in $\F_9$. 
\end{example}

To conclude this section, we give alternative proof of the following theorem concerning the minimum value of the TAV order function $\ord\colon\K\to\N\cup\{+\infty\}$ (see also \cite[Proposition 2.8]{IMS23-1}). 

\begin{theorem}[{\cite[Theorem 1.2(i)]{IMS23-1}}]\label{thm:24}
For any knot $K$, we have $\ord(K)\geq 24$. 
\end{theorem}

\begin{proof}
It is known that $35$ finite groups are of weight one, out of $59$ finite groups of order less than $24$ (see \cite{GN}). 
We can apply (\ref{eq:cyclic}) for $23$ abelian finite groups, 
Corollary \ref{cor:dihedral2} for $D_3,D_5,D_7,D_9,D_{11}$, 
Example \ref{ex:A4} for $A_4$, 
Example \ref{ex:Dic} for ${\rm Dic}_3, {\rm Dic}_5$, and 
Example \ref{ex:Gmpk} for $C_5 \rtimes C_4 = G(4,5|2), C_7 \rtimes C_3 = G(3,7|2)$. 
They follow that the twisted Alexander polynomials are nonzero for any knots. Then the remaining groups are only $C_3\rtimes D_3$ and $C_3\times D_3$ out of $59$ finite groups. 

Let $G_1=C_3\rtimes D_3=C_3^2\rtimes_2C_2$, where $G_1'=C_3^2$ and $G_1/G_1'=C_2$. Using 
Corollary \ref{cor:modp-TAP} for $p=3,n=2,k=0,l=2,d=1$, and $\zeta=2\in\F_3$,  we have  
\begin{equation}\label{eq:cyclic9}
\ov{\D_K^{\rho\circ f}}(t)
=
\left(
\frac{\ov{\D_K}(t)}{t-1}\cdot
\frac{\ov{\D_K}(2t)}{2t-1}\right)^{9}
\end{equation}
in $\F_3(t)$. 
Similarly, let $G_2=C_3\times D_3=C_3\rtimes C_6$, where $G_2'=C_3$ and $G_2/G_2'=C_6$. Using Corollary 
\ref{cor:modp-TAP} for $p=3,n=1,k=1,l=2,d=1$, and $\zeta=2\in\F_3$, we have  
\begin{equation}\label{eq:cyclic10}
\ov{\D_K^{\rho\circ f}}(t)
=
\left(
\frac{\ov{\D_K}(t)}{t-1}\cdot
\frac{\ov{\D_K}(2t)}{2t-1}
\right)^9
\end{equation}
in $\F_3(t)$. Hence, the twisted Alexander polynomials for these two finite groups are nonzero for any knots. 

This completes the proof of Theorem \ref{thm:24}. 
\end{proof}

\begin{remark}
In the proof of Theorem \ref{thm:24}, 
$G_2$ can also be expressed as $C_3^2 \rtimes C_2$. 
Then by applying Corollary \ref{cor:modp-TAP} with respect to this semi-direct product, namely, 
for $p=3,n=2,k=0,l=2,d=1$, and $\zeta=2\in\F_3$,  we obtain the same formula  (\ref{eq:cyclic10}). 
Moreover, we can take $C_9, C_3^2, C_3^2$ as normal subgroups of $D_9, G_1, G_2$ respectively, 
then all the orders are the same $9=3^2$ and all the quotients are the same $C_2$. 
Hence the formulas (\ref{eq:cyclic9}), (\ref{eq:cyclic10}), and Corollary \ref{cor:dihedral2} when $p^n=3^2$ coincide, 
even though they have different commutator subgroups. 
Such phenomena commonly arise in mod $p$ formulas. 
\end{remark}

\section{Twisted Alexander polynomials for a central extension}\label{sec:extension}

In this section, we provide a formula for the twisted Alexander polynomial of a central extension of a finite group by a cyclic group. Using the formula, we can find an important fact about the minimality of the TAV groups that appear in a series of central extensions (see Corollary \ref{cor:smallest}).  

Let $G_{k,1}$ be a finite group with the abelianization $C_k=\la \bar{x}\,|\,\bar{x}^k\ra$, namely, it is of weight one and fits the following short exact sequence: 
$$
1\longrightarrow H \longrightarrow G_{k,1} \overset{\pi}{\longrightarrow} C_{k} \longrightarrow 1,
$$
where $H$ is the commutator subgroup of $G_{k,1}$. We then define a group $G_{k,n}$ as the pull-back of two epimorphisms $\pi\colon G_{k,1}\to C_k$ and $\pi_k\colon C_{kn}=\la x\,|\,x^{kn}\ra\to C_k$, that is, $G_{k,n}=\{(z,x^l)\in G_{k,1}\times C_{kn}\,|\,\pi(z)=\pi_k(x^l)\}$:
\begin{equation*}
\begin{CD}
1@>>> H @>>> G_{k,n} @>{\pr_2}>> C_{kn} @>>> 1\\
@. @VV{\mathrm{id}_H}V @VV{\pr_1=\,\pr}V @VV{\pi_k}V \\
1@>>> H @>>> G_{k,1} @>{\pi}>> C_{k} @>>> 1
\end{CD}
\end{equation*}
where $\pr_j$ is the projection onto the $j$-th  component. We note that $\pr_2\colon G_{k,n}\to C_{kn}$ is the abelianization, and the inclusion $\iota\colon C_n=\la y\,|\,y^n\ra\to G_{k,n},\,y\mapsto (e,x^k)$ is a homomorphism such that 
$$
1\longrightarrow C_n \overset{\iota}{\longrightarrow} G_{k,n}\overset{\pr}{\longrightarrow} G_{k,1}\longrightarrow 1
$$ 
is a central extension. Note that the weight of $G_{k,n}$ is one. 
We call the group $G_{k,n}$ the \textit{$n$-th central extension} of $G_{k,1}$. If the group $G_{k,1}$ does not appear as the $n$-th central extension $(n\geq2)$ of any other group, we call it a \textit{seed} of any series of central extensions. At the end of this section, we discuss a characterization of seed. 

Now, we prepare the following lemma for application. 

\begin{lemma}\label{lem:lifting}
For a given epimorphism $f_1\colon G(K) \to G_{k,1}$, there exists an epimorphism $f_n\colon G(K) \to G_{k,n}$ such that $\pr\circ f_n=f_1$. 
\end{lemma}

\begin{proof}
Let $\bar{\phi}\colon G(K)\to C_{kn}$ be the composition of the abelianization $\phi\colon G(K)\to\Z$ and the projection $\Z\to\Z/kn\Z\cong C_{kn}$. 
We define a lift $f_n\colon G(K)\to G_{k,n}$ to be $f_n(u)=(f_1(u),\bar{\phi}(u))$ for $u\in G(K)$. It is clearly a homomorphism and satisfies $\pr\circ f_n=f_1$. Since $f_1(G(K)') = H$, we have $f_n(G(K)') = H \times \{e\} = G_{k,n}'$.
Also, since $\bar{\phi}$ is an epimorphism, there is an element $u \in
G(K)$ such that $\bar{\phi}(u) = {\rm pr}_2 \circ f_n(u) \in C_{kn}$ is a generator of $C_{kn}$. Because ${\rm pr}_2 \colon G_{k,n} \to C_{kn}$ is the abelianization, $G_{k,n}$ is generated by $G_{k,n}' = f_n(G(K)')$ and $f_n(u)$, which means that $f_n$ is surjective.
\end{proof}

Let us assume that there exists an epimorphism $f_n\colon G(K) \to G_{k,n}$. Then we have the following formula for the twisted Alexander polynomial associated with the regular representation $\rho_n\colon G_{k,n}\to \mathrm{GL}(kn|H|,\C)$. For simplicity, let us denote by $\tilde{\rho}_1$ the composition of $\pr\colon G_{k,n}\to G_{k,1}$ and the regular representation $\rho_1\colon G_{k,1}\to \mathrm{GL}(k|H|,\C)$; i.e., $\tilde{\rho}_1=\rho_1\circ\mathrm{pr}$. Moreover, let $i=\sqrt{-1}$. 

\begin{theorem}[Theorem \ref{thm:main-2}]\label{thm:central}
$\displaystyle{
\D_K^{\rho_n\circ f_n}(t)
=
\prod_{j=0}^{n-1}
\D_K^{\tilde{\rho}_1 \circ f_n}\big(e^{\frac{2\pi i}{kn}j}t\big).}$
\end{theorem}

Namely, the twisted Alexander polynomial of the $n$-th central extension $G_{k,n}$ is determined by that of $G_{k,1}$. 

\begin{remark}
Similar to Theorem \ref{modp-thm} in the previous section, 
Theorem \ref{thm:central} also holds for a finite CW complex $X$ and $\phi\in H^1(X;\Z)=\mathrm{Hom}(\pi_1(X),\Z)$ an epimorphism, which is non-necessarily abelianization. 
\end{remark}

\begin{definition}
Let $G$ be a finite group. 
We define $\MV(G)$ to be the set of knots $K$ such that $K$ admits $G$ as a TAV group. 
\end{definition}

Note that $\MV(G)\neq\emptyset$ if and only if $G$ is a TAV group. Using Lemma \ref{lem:lifting} and Theorem \ref{thm:central}, we can obtain the following corollary. 

\begin{corollary}\label{cor:smallest}
The following two conditions are equivalent:
\begin{itemize}
\item[(i)]
There exists an epimorphism $f_1\colon G(K) \to G_{k,1}$ such that $\D_K^{\rho_1\circ f_1}(t)=0$.
\item[(ii)]
There exists an epimorphism $f_n\colon G(K) \to G_{k,n}$ such that $\D_K^{\rho_n\circ f_n}(t)=0$.
\end{itemize}
In particular, the set $\mathcal{V}(G_{k,1})$ coincides with $\mathcal{V}(G_{k,n})$ for any $n\geq 2$. Accordingly, if $G_{k,n}~(n\geq2)$ is a TAV group, then it is not the smallest for any non-fibered knots. 
\end{corollary}

Before proving Theorem \ref{thm:central}, let us see some examples. 

\begin{example}\label{ex:extension}
(i) Let $H=C_q$ be the cyclic group of odd order $q$ and $k=2$. Then the group $C_q\rtimes C_{2n}$ is a central extension of the dihedral group $D_q=C_q\rtimes C _2$:
$$
1\longrightarrow C_n \longrightarrow C_q\rtimes C_{2n} \overset{\pr}{\longrightarrow} D_q \longrightarrow 1.
$$
In particular, when $n=2$, $C_q\rtimes C_4$ is the dicyclic group $\mathrm{Dic}_q$ appeared in Sections \ref{sec:2} and \ref{sec:modp}. Using Theorem \ref{thm:central}, we have 
$$
\D_K^{\rho_2\circ f_2}(t)
=
\D_K^{\tilde{\rho}_1 \circ f_2}(t)\cdot
\D_K^{\tilde{\rho}_1 \circ f_2}\left(it\right).
$$
Similarly, when $n=3$, 
$$
\D_K^{\rho_3\circ f_3}(t)
=
\D_K^{\tilde{\rho}_1 \circ f_3}(t)\cdot
\D_K^{\tilde{\rho}_1 \circ f_3}\big(e^{\frac{\pi i}{3}}t\big)\cdot
\D_K^{\tilde{\rho}_1 \circ f_3}\big(e^{\frac{2\pi i}{3}}t\big)
$$
holds. Hence, Corollary \ref{cor:smallest} implies that 
$\mathcal{V}(D_q)=\mathcal{V}(\mathrm{Dic}_q)=\mathcal{V}(C_q\rtimes C_6)$ holds.  

(ii) Let $H$ be the alternating group $A_4$. Then we have a central extension 
of the symmetric group $S_4=A_4\rtimes C_2$:
$$
1\longrightarrow C_n \longrightarrow A_4\rtimes C_{2n} \overset{\pr}{\longrightarrow} S_4 \longrightarrow 1.
$$
Thus, Corollary \ref{cor:smallest} implies that 
$\mathcal{V}(S_4)=\mathcal{V}(A_4\rtimes C_{2n})$ holds for $n\geq 2$. 

(iii) Let $H$ be $\mathrm{SL}_2(\F_3)$, and $G_{k,1}$ a non-split extension of $C_2$ by $\mathrm{SL}_2(\F_3)$; i.e. $G_{k,1}=\mathrm{SL}_2(\F_3).C_2$. Then we have a central extension $G_{k,n}=\mathrm{SL}_2(\F_3).C_{2n}$ of $G_{k,1}$: 
$$
1\longrightarrow C_{n} \longrightarrow G_{k,n} \overset{\pr}{\longrightarrow} G_{k,1} \longrightarrow 1.
$$
Hence, Corollary \ref{cor:smallest} implies that $\mathcal{V}(G_{k,1})=\mathcal{V}(G_{k,n})$ holds for $n\geq2$. 
\end{example}

Let $\eta_j\colon C_{kn}\to \mathrm{GL}(1,\C)$ be a one-dimensional representation defined by $\eta_j(x)=\exp\big(\frac{2\pi i}{kn}j\big)$, and $\tau_j\colon G_{k,n}\to \mathrm{GL}(1,\C)$ the composition of $\pr_2\colon G_{k,n}\to C_{kn}$ and $\eta_j$. We then define a representation $\tau_{n,j}\colon G_{k,n}\to \mathrm{GL}(k|H|,\C)$ by $\tau_{n,j}=\tau_j\otimes \tilde{\rho}_1$. It is easy to see that 
$$
\tau_{n,j}(z,x^l)=(\tau_j\otimes\tilde{\rho}_1)(z,x^l)
=\tau_j(z,x^l)\otimes\tilde{\rho}_1(z,x^l)
=\eta_j(x^l)\otimes\rho_1(z)
$$
hold for $(z,x^l)\in G_{k,n}$. 

\begin{lemma}\label{pro:regular2}
The representation $\bigoplus_{j=0}^{n-1}\tau_{n,j}$ is conjugate to 
the regular representation $\rho_{n}\colon G_{k,n}\to \mathrm{GL}(kn|H|,\C)$. 
\end{lemma}

\begin{proof}
Note that for a finite group $G$, a given representation $\rho\colon G\to\mathrm{GL}(|G|,\C)$ is conjugate to the regular representation of $G$ if and only if the character $\chi_\rho$ satisfies $\chi_\rho(x)=|G|$, if $x=e$, and $\chi_\rho(x)=0$, otherwise. 
For $(z,x^l)\in G_{k,n}~(0\leq l\leq kn-1)$, we obtain 
\begin{align*}
\chi_{\bigoplus_{j=0}^{n-1}\tau_{n,j}}(z,x^l)
&=
\sum_{j=0}^{n-1}\chi_{\tau_{n,j}}(z,x^l)
=
\sum_{j=0}^{n-1}\mathrm{tr}\big(\tau_{n,j}(z,x^l)\big)\\
&=
\sum_{j=0}^{n-1}\mathrm{tr}\big(\eta_j(x^l)\otimes\rho_1(z)\big)\\
&=
\sum_{j=0}^{n-1}\mathrm{tr}\big(\large(e^{\frac{2\pi i}{kn}j}\large)^l\cdot \rho_1(z)\big)
=
\sum_{j=0}^{n-1}
e^{\frac{2\pi i}{kn}jl}\cdot\mathrm{tr}\large(\rho_1(z)\large)\\
&=
\begin{cases}
\displaystyle{
\sum_{j=0}^{n-1}e^{\frac{2\pi i}{n}j\bar{l}}\cdot k|H|
}
, & \text{if $z=e$ and $l=k\bar{l}~(0\leq \bar{l}\leq n-1)$}\\
0, & \text{otherwise}
\end{cases}\\
&=
\begin{cases}
kn|H|, & \text{if $z=e$ and $l=0$ by Lemma \ref{lem:sum} below}\\
0, & \text{otherwise}.
\end{cases}
\end{align*}
This completes the proof of Lemma \ref{pro:regular2}.
\end{proof}

\begin{lemma}\label{lem:sum}
$\displaystyle{
\sum_{j=0}^{n-1}e^{\frac{2\pi i}{n}dj}
=
\begin{cases}
n,& d=0\\
0,& 0<d<n
\end{cases}
}$
\end{lemma}

\begin{proof}
The case $d=0$ is clear. When $0<d<n$, we have $\big(\large(e^{\frac{2\pi i}{n}j}\large)^d\big)^n=e^{2\pi idj}=1$ and $\large(e^{\frac{2\pi i}{n}j}\large)^d\not=1$. 
In general, suppose that $(\eta^d)^n=1$ and $\eta^d\not=1$, then 
we obtain $\sum_{j=0}^{n-1}(\eta^d)^j=0$ because $(\eta^d-1)\big((\eta^d)^{n-1}+(\eta^d)^{n-2}+\cdots+\eta^d+1\big)=0$ holds.
\end{proof}

\begin{proof}[Proof of Theorem \ref{thm:central}] 
By Lemma \ref{pro:regular2}, we see that the composition $\rho_n\circ f_n$ is conjugate to 
\begin{align*}
\big(
\bigoplus_{j=0}^{n-1}
\tau_{n,j}
\big)\circ f_n
=
\big(
\bigoplus_{j=0}^{n-1}
\tau_j\otimes\tilde{\rho}_1
\big)\circ f_n
=
\bigoplus_{j=0}^{n-1}(\tau_j\circ f_n)\otimes
(\tilde{\rho}_1 \circ f_n).
\end{align*}
Hence, 
\begin{align*}
\D_K^{\rho_n\circ f_n}(t)
&=
\prod_{j=0}^{n-1}\D_K^{(\tau_j\circ f_n)\otimes(\tilde{\rho}_1 \circ f_n)}(t)
=
\prod_{j=0}^{n-1}\D_K^{\tilde{\rho}_1 \circ f_n}\big(
e^{\frac{2\pi i}{kn}j}t\big)
\end{align*}
holds as desired.
\end{proof}

Finally we end this section with a characterization of seed of a series of central extensions. 

\begin{proposition}\label{pro:seed}
Let $G$ be a finite group of weight one. Then $G$ is not a seed if and only if there exists a nontrivial cyclic group $C$ in the center $Z(G)$ such that $C\cap G'=\{e\}$. 
\end{proposition}

\begin{proof}
Assume that $G=G_{k,n}~(n\geq 2)$. Let $C_n$ be a subgroup of $G_{k,n}$ generated by $(e,x^k)$. Then $C_n\subset Z(G_{k,n})$ and $C_n\cap H=\{e\}$ hold, where $H$ is the commutator subgroup of $G_{k,n}$. 

Conversely, if there exists a nontrivial cyclic group $C$ in $Z(G)$ such that $C\cap G'=\{e\}$, we can define the group $G_{k,1}$ by $G/C$. Namely, $G$ is a central extension of $G_{k,1}$ by $C$. 
\end{proof}

\section{Concluding remark}\label{sec:remark}

As mentioned in the introduction, the TAV order of a knot induces the TAV order function $\ord\colon\K\to\N\cup\{+\infty\}.$
Using Theorem \ref{thm:main-4}, we know that the image of the restriction of the TAV order function $\ord|_{\mathcal{N}} \colon \mathcal{N}\to\mathbb{N}$ is included in the set 
$$
\{24,30,42,48,60,66,70,72,78,84,90,96,102,110,114,120,126,\ldots\}
$$
consisting of orders of TAV groups $G$ (see \cite[Section 5]{IMS23-1}). However, it is not known whether these values are actually realized as the TAV orders of non-fibered knots. We then proposed the following problem in \cite{IMS23-1}. 

\begin{problem}[{\cite[Problem~5.4]{IMS23-1}}]\label{problem:image}
Determine the image of the TAV order function $\ord|_{\mathcal{N}} \colon \mathcal{N}\to\mathbb{N}$. 
\end{problem}

As shown in Corollary \ref{cor:smallest} and Example \ref{ex:extension}, there are several series of TAV groups $G_{k,n}$ such that $\mathcal{V}(G_{k,1})=\mathcal{V}(G_{k,n})$ holds for $n\geq2$. Thus, the TAV group $G_{k,n}~(n\geq2)$ is not the smallest for any non-fibered knots. Using this fact, we can show the following as a first step to solve  Problem \ref{problem:image}. 

\begin{corollary}[{Corollary \ref{cor:image}}]\label{cor:dic33}
The set of orders of TAV groups contains $\mathrm{Im}\,\ord|_{\mathcal{N}}$ as a proper subset.
\end{corollary}

\begin{proof}
Let us consider the dicyclic group $G=\mathrm{Dic}_{33}$. This is a central extension of the dihedral group $D_{33}$ and the unique TAV group of order $132$ (see 
\cite[Theorem 3.1(iii)]{IMS-III}). Hence, we can conclude that the order $132$ is \textit{not} included in the image of the TAV order function $\ord|_{\mathcal{N}} \colon \mathcal{N}\to\mathbb{N}$, because $\mathcal{V}(\mathrm{Dic}_{33})=\mathcal{V}(D_{33})$ holds by Corollary \ref{cor:smallest}. This completes the proof. 
\end{proof}

The formulas for twisted Alexander polynomials of a knot for group extensions include a much richer content than we have presented here. A more detailed study will be the subject of our future research. 

\subsection*{Acknowledgments}
The authors are supported in part by 
JSPS KAKENHI Grant Numbers JP23K20799, 
JP25K07012, and JP25K17252. 

\bibliographystyle{amsplain}

\end{document}